\documentclass[11pt,a4paper]{article}
\usepackage{amsmath}
\usepackage{amsthm}
\usepackage{amsfonts}
\usepackage{latexsym}
\usepackage{geometry}
\newtheorem{thm}{Theorem}[section]

\newtheorem{lem}[thm]{Lemma}

\newcommand{\be}{\begin{equation}}
\newcommand{\ee}{\end{equation}}
\newcommand{\ben}{\begin{enumerate}}
\newcommand{\een}{\end{enumerate}}
\newcommand{\pa}{{\partial}}

\newcommand{\pxi}{{\pa \over \pa x^i}}

\vspace{7cm} \textwidth 14cm \textheight 21.6cm
\title{Generalized P-Reducible $(\alpha, \beta)$-Metrics\\ with Vanishing S-curvature\footnote{Annales
Polonici Mathematici, {\bf 114}(1)  (2015), 67-79.}}
\author{A. Tayebi and H. Sadeghi}
\begin{document}

\maketitle
\begin{abstract}
In this paper, we study one of the  open problems in Finsler geometry  which presented by Matsumoto-Shimada  about the existence  of P-reducible  metric which is not C-reducible. For this aim, we study a class of Finsler metrics called generalized P-reducible metrics that  contains the class of P-reducible metrics. We prove that every  generalized P-reducible $(\alpha, \beta)$-metric with vanishing S-curvature reduces to a Berwald metric or C-reducible metric.  It results that there is not any concrete P-reducible $(\alpha,\beta)$-metric with vanishing S-curvature.
\end{abstract}

\section{Introduction}
In 1975,  the well-known Physicist Y. Takano  wrote a paper on physics  which considered the  field  equation  in a Finsler  space  and  proposed  certain  geometrical  problems  in  Finsler  geometry \cite{Tak}. He  requested   mathematicians to  find  some  interesting  special  forms  of  hv-curvature from  the  standpoint  of  physics. In 1978,   Matsumoto introduced the  notion of P-reducible Finsler metrics as an answer to Takano which were  a  generalization  of  C-reducible Finsler metrics \cite{M}. For a   Finsler metric of dimension $n\geq 3$, he found some conditions under which the Finsler metric was P-reducible.

Since the study  of  hv-curvature becomes  urgent  necessity  for  the  Finsler  geometry  as  well  as  for  theoretical physics, then  Matsumoto-Shimada study the curvature properties of P-reducible metrics in \cite{MShi}. They  introduced the following open problem:
\begin{center}
\emph{Is there any concrete P-reducible metric which is not C-reducible?}
\end{center}

In \cite{MaHo}, Matsumoto-H\={o}j\={o} proves that  $F$ is C-reducible  if and only if it is a Randers  metric or Kropina metric. These metrics  defined by $F=\alpha+\beta$ and $F=\alpha^2/\beta$, respectively,  where $\alpha=\sqrt{a_{ij}y^iy^j}$ is a Riemann metric and $\beta:=b_i(x) y^i$ a 1-form on a manifold $M$. The Randers metrics were introduced by G. Randers in the context of general relativity  and have been widely applied in many areas of natural science, including biology, ecology,  physics and psychology, etc  \cite{CS}. The Kropina metric was  introduced by L. Berwald  in connection with a two-dimensional Finsler space with rectilinear extremal \cite{Ber}.

In \cite{Numata}, Numata introduced an interesting  family of Finsler metrics called the Numata-type metrics.  The Numata-type Finsler metrics are defined by $F:=\bar{F}+\eta$, where   $\bar{F}(y)=\sqrt{g_{ij}(y)y^iy^j}$ is a locally Minkowskian metric  and $\eta=\eta_i(x)y^i$  a closed one-form on a manifold $M$. $F$ is called a Randers change of ${\bar F}$. By a simple calculation, we get
\[
C_{ijk}={\bar C}_{ijk}+\frac{1}{2\bar{F}}\Big\{h_{ij}D_m+h_{jk}D_i+h_{ki}D_j\Big\},
\]
where $D_i:=\eta_i-\eta y_i/(\bar{F})^2$ and $h_{ij}:=FF_{ij}$ is the angular metric. Define $\eta_{i|j}$ by $\eta_{i|j}\gamma^j:=d\eta_i-\eta_j\gamma^j_i$, where $\gamma^i:=dx^i$ and $\gamma^j_i:=\Gamma^j_{ik}dx^k$ denote
the linear connection form of $\bar{F}$. Put
\[
\mathfrak{D}_{ij}:=\frac{1}{2}(\eta_{i|j}+\eta_{j|i}).
\]
Then the Landsberg curvature of $F$ is given by
\be
L_{ijk}=\lambda C_{ijk}+ a_ih_{jk}+a_jh_{ki}+a_kh_{ij},\label{GP-red}
\ee
where
\[
\lambda=\frac{1}{2F}\mathfrak{D}_{ij}y^iy^j\ \ \ \  \textrm{and}\  \ \ \ a_i:=\frac{1}{2F^2{\bar F}^3}\Big[2F{\bar F}^2\mathfrak{D}_{ik}-2F\mathfrak{D}_{kl}y^ly_i-(1+{\bar F}^2)\mathfrak{D}_{kl}y^lD_j \Big]y^k.
\]

A Finsler metric $F$ is called \emph{generalized P-reducible} if its Landsberg curvature is given by (\ref{GP-red}), where $a_i=a_i(x, y)$ and $\lambda=\lambda(x, y)$ are scalar function on $TM$. Thus every Numata-type metric is a generalized P-reducible metric. By definition, if $a_i=0$  then $F$  reduces to a general relatively  isotropic Landsberg metric and if $\lambda=0$ then $F$ is P-reducible. Thus the study of this class of Finsler spaces will enhance our understanding of the geometric meaning of P-reducible metrics.

The notion of  S-curvature is originally introduced by Shen for the volume comparison theorem \cite{Shvol}.  The Finsler metrics with vanishing  S-curvature are  some of important geometric structures which deserved to be studied deeply. Then it is a natural problem to study Finsler metrics with vanishing S-curvature \cite{TN}.

An $(\alpha, \beta)$-metric is a Finsler metric of the form $F:=\alpha \phi(s)$, $s=\beta/\alpha$, where  $\phi=\phi(s)$  is a $C^\infty$ on $(-b_0, b_0)$, $\alpha=\sqrt{a_{ij}(x)y^iy^j}$ is a Riemannian metric and $\beta =b_i(x)y^i$ is a 1-form on  $M$. For example, $\phi= c_1\sqrt{1+c_2s^2}+c_3s$  is called Randers type metric, where  $c_1>0$, $c_2$ and $c_3$ are constant. In this paper, we  characterize  generalized P-reducible  $(\alpha,\beta)$-metrics with vanishing S-curvature and  prove the following.

\begin{thm}\label{mainthm}
Let $F=\alpha\phi(s)$, $s=\beta/\alpha$, be an  $(\alpha,\beta)$-metric on a manifold $M$. Suppose that $F$  is a generalized P-reducible metric with vanishing S-curvature. Then $F$ is a Berwald metric or  C-reducible metric.
\end{thm}

By Theorem \ref{mainthm}, it follows that there is no concrete P-reducible $(\alpha,\beta)$-metric with vanishing S-curvature (see Lemma \ref{lem5}).

\bigskip

In this paper, we use the Berwald connection  and the $h$- and $v$- covariant derivatives of a Finsler tensor field are denoted by `` $|$ " and ``, " respectively.

%------------------------------------------------------------------------------------------------------------
\section{Preliminaries}\label{sectionP}
%------------------------------------------------------------------------------------------------------------
Let $(M, F)$ be a Finsler manifold. Suppose that  $x\in M$ and $F_x:=F|_{T_xM}$. We define ${\bf C}_y:T_xM\otimes T_xM\otimes
T_xM\rightarrow \mathbb{R}$ by ${\bf C}_y(u,v,w):=C_{ijk}(y)u^iv^jw^k$ where
\[
C_{ijk}:=\frac{1}{2}\frac{\partial g_{ij}}{\partial y^k}=\frac{1}{4}\frac{\partial^3 F^2}{\partial y^i\partial y^j\partial y^k},
\]
and $g_{ij}:=\frac{1}{2}[F^2]_{y^iy^j}$. The family ${\bf C}:=\{{\bf C}_y\}_{y\in TM_0}$  is called the Cartan torsion. It is well known that ${\bf{C}}=0$ if and only if $F$ is Riemannian. For $y\in T_x M_0$, define  mean Cartan torsion ${\bf I}_y$ by ${\bf I}_y(u):=I_i(y)u^i$, where $I_i:=g^{jk}C_{ijk}$.

For   $y \in T_xM_0$, define the  Matsumoto torsion ${\bf M}_y:T_xM\otimes T_xM \otimes T_xM \rightarrow \mathbb{R}$ by ${\bf M}_y(u,v,w):=M_{ijk}(y)u^iv^jw^k$ where
\[
M_{ijk}:=C_{ijk} - {1\over n+1}  \big\{ I_i h_{jk} + I_j h_{ik} + I_k h_{ij}\big \},\label{Matsumoto}
\]
and $h_{ij}=g_{ij}-F_{y^i}F_{y^j}$ is the angular metric.  $F$ is said to be C-reducible if ${\bf M}_y=0$.
\begin{lem}\label{MaHo}{\rm (\cite{MaHo})}
\emph{A  Finsler metric $F$ on a manifold $M$ of dimension $n\geq 3$ is a Randers  metric or Kropina metric if and only if\  ${\bf M}_y =0$, $\forall y\in TM_0$.}
\end{lem}

A Finsler metric called semi-C-reducible if its Cartan tensor is given by following
\begin{eqnarray}
C_{ijk}=\frac{p}{n+1}\big\{h_{ij}I_k+h_{jk}I_i+h_{ik}I_j\big\}+\frac{q}{\|\textbf{I}\|^2}I_iI_jI_k,\label{m1}
\end{eqnarray}
where $p=p(x,y)$ and $q=q(x,y)$ are scalar function on $TM$ satisfying $p+q=1$ and $\|\textbf{I}\|^2=I^mI_m$ (see \cite{Mat}\cite{TPN}\cite{TS}).
\begin{lem}{\rm (\cite{Mat})}
\emph{Every non-Riemannian $(\alpha,\beta)$-metric on a manifold $M$ of dimension $n\geq 3$ is semi-C-reducible.}
\end{lem}

The horizontal covariant derivatives of the Cartan torsion ${\bf C}$ and mean Cartan torsion ${\bf I}$ along geodesics give rise to  the  Landsberg curvature  ${\bf L}_y:T_xM\otimes T_xM\otimes T_xM\rightarrow \mathbb{R}$  and mean Landsberg curvature   ${\bf J}_y:T_xM\rightarrow \mathbb{R}$ which are defined by ${\bf L}_y(u,v,w):=L_{ijk}(y)u^iv^jw^k$ and ${\bf J}_y(u):=J_i(y)u^i$, respectively,  where
\[
L_{ijk}:=C_{ijk|s}y^s, \ \ \ \ \  J_i:=I_{i|s}y^s.
\]
The families  ${\bf L}:=\{{\bf L}_y\}_{y\in TM_{0}}$  and ${\bf J}:=\{{\bf J}_y\}_{y\in TM_{0}}$ are called the Landsberg curvature and mean Landsberg curvature, rspectively. A Finsler metric is called a Landsberg metric and weakly Landsberg metric  if ${\bf{L}}=0$ and ${\bf J}=0$, respectively.

A Finsler metric $F$ on $n$-dimensional manifold $M$ is called P-reducible if its  Landsberg curvature is given by following
\[
L_{ijk}={1\over n+1}  \big\{ J_i h_{jk} + J_j h_{ik} + J_k h_{ij}\big \}.
\]
It is easy to see that, every C-reducible metric is P-reducible. But the converse maybe is not true \cite{M3}.

\bigskip

Given an $n$-dimensional Finsler manifold $(M,F)$, then a global vector field ${\bf G}$ is induced by $F$ on $TM_0$, which in a standard coordinate $(x^i,y^i)$ for $TM_0$ is given by ${\bf G}=y^i {{\partial} \over {\partial x^i}}-2G^i(x,y){{\partial} \over {\partial y^i}}$, where $G^i=G^i(x, y)$ are called spray coefficients and given by following
\[
G^i:=\frac{1}{4}g^{il}\Big[\frac{\partial^2 F^2}{\partial x^k\partial y^l}y^k-\frac{\partial F^2}{\partial x^l}\Big],\ \ \ \  y\in T_xM.
\]
${\bf G}$ is called the  spray associated  to $F$.

For  $y \in T_xM_0$, define ${\bf B}_y:T_xM\otimes T_xM \otimes T_xM\rightarrow T_xM$  by ${\bf B}_y(u, v, w):=B^i_{\ jkl}(y)u^jv^kw^l{{\partial } \over {\partial x^i}}|_x$, where
\[
B^i_{\ jkl}:={{\partial^3 G^i} \over {\partial y^j \partial y^k \partial y^l}}.
\]
$\bf B$ is called the Berwald curvature and $F$ is called a Berwald metric if $\bf{B}=0$.

\bigskip

For an $(\alpha, \beta)$-metric, let us define $b_{i|j}$ by $b_{i|j}\theta^j:=db_i-b_j\theta^j_i$, where $\theta^i:=dx^i$ and $\theta^j_i:=\Gamma^j_{ik}dx^k$ denote
the Levi-Civita connection form of $\alpha$. Let
\begin{eqnarray*}
&&r_{ij}:=\frac{1}{2}(b_{i|j}+b_{j|i}), \ \ \ s_{ij}:=\frac{1}{2}(b_{i|j}-b_{j|i}),\\
&&r_{i0}: = r_{ij}y^j, \  \ r_{00}:=r_{ij}y^iy^j, \ \ r_j := b^i r_{ij},\\
&&s_{i0}:= s_{ij}y^j, \  \ \ s_j:=b^i s_{ij}, \ \ \  r_0:= r_j y^j,\ \  \ \  s_0 := s_j y^j.
\end{eqnarray*}

Let $G^i=G^i(x,y)$ and $G^i_{\alpha}=G^i_{\alpha}(x,y)$ denote the
coefficients  of $F$ and  $\alpha$ respectively in the same coordinate system. The following holds
\begin{eqnarray}
G^i=G^i_{\alpha}+\alpha Q s^i_0+(-2Q\alpha s_0+r_{00})(\Theta\frac{y^i}{\alpha}+\Psi b^i),\label{G1}
\end{eqnarray}
where
\begin{eqnarray*}
&&Q:=\frac{\phi'}{\phi-s\phi'}, \ \ \ \Delta:=1+sQ+(b^2-s^2)Q',\\
&&\Theta:=\frac{Q-sQ'}{2\Delta}, \ \ \ \  \Psi:=\frac{Q'}{2\Delta}.
\end{eqnarray*}

The mean Landsberg curvature of an $(\alpha,\beta)$-metric $F =\alpha\phi(s)$, is given by following formula
\begin{eqnarray}
\nonumber J_i:=\!\!\!\!&-&\!\!\!\!\frac{1}{2\alpha^4\Delta}\Bigg(\frac{2\alpha^2}{b^2-s^2}\Big[\frac{\Phi}{\Delta}+(n+1)(Q-sQ')\Big](r_0+s_0)h_i
\\ \nonumber \!\!\!\!&+&\!\!\!\! \frac{\alpha}{b^2-s^2}\Big[\Psi_1+s\frac{\Phi}{\Delta}\Big](r_{00}-2\alpha Q s_0)h_i+\alpha\Big[-\alpha Q's_0h_i+\alpha Q(\alpha^2 s_i-\bar{y}_is_0)\\
\!\!\!\!&+&\!\!\!\! \alpha^2\Delta s_{i0}+\alpha^2(r_{i0}-2\alpha Q s_0)-(r_{00}-2\alpha Q s_0)\bar{y}_i\Big]\frac{\Phi}{\Delta}\Bigg),\label{meanlanndsberg}
\end{eqnarray}
where
\begin{eqnarray*}
&&\Psi_1:=\sqrt{b^2-s^2}\Delta^{\frac{1}{2}}\Big[\frac{\sqrt{b^2-s^2}}{\Delta^{\frac{3}{2}}}\Big]',\\
&&h_i:=\alpha b_i-s \bar{y}_i, \ \ \  \bar{y}_i:=a_{ij}y^j,\\
&&\Phi:=-(Q-sQ')(n\Delta+1+sQ)-(b^2-s^2)(1+sQ)Q''.
\end{eqnarray*}
For more details, see \cite{Cheng}. Then we have
\begin{eqnarray}
\bar{J}:=b^iJ_i=-\frac{1}{2\alpha^2\Delta}\Big\{\Psi_1(r_{00}-2\alpha Q s_0)+\alpha\Psi_2(r_0+s_0)\Big\},\label{05}
\end{eqnarray}
where
\[
\Psi_2:=2(n+1)(Q-sQ')+3\frac{\Phi}{\Delta}.
\]

\bigskip

For a Finsler metric $F$ on an $n$-dimensional manifold $M$, the
Busemann-Hausdorff volume form $dV_F = \sigma_F(x) dx^1 \cdots
dx^n$ is defined by
\[
\sigma_F(x) := {{\rm Vol} (\Bbb B^n(1))
\over {\rm Vol} \Big \{ (y^i)\in \mathbb{R}^n \ \big | \ F \big ( y^i
\pxi|_x \big ) < 1 \Big \} }.
\]
Let $G^i(x, y)$ denote the geodesic coefficients of $F$ in the same
local coordinate system. The S-curvature is defined by
\[
 {\bf S}({\bf y}) := {\pa G^i\over \pa y^i}(x,y) - y^i {\pa \over \pa x^i}
\Big [ \ln \sigma_F (x)\Big ],
\]
where ${\bf y}=y^i\pxi|_x\in T_xM$. If $F$ is a Berwald metric then ${\bf S}=0$.

\bigskip

In \cite{ChSh3}, Cheng-Shen  characterize $(\alpha,\beta)$-metrics with isotropic
S-curvature.
\begin{lem}\label{CS0}{\rm (\cite{ChSh3})}
\emph{Let $F=\alpha\phi(s)$ , $s=\beta/\alpha$, be an non-Riemannian $(\alpha,\beta)$-metric on a manifold $M$ of dimension $n\geq 3$ and $b:=\|\beta_x\|_{\alpha}$. Suppose that $F$ is not a Finsler metric of Randers type. Then $F$ is of isotropic S-curvature, $\textbf{S}=(n+1)cF$, if and
only if
one of the following holds\\
(a) $\beta$ satisfies
\begin{eqnarray}
r_{ij}=\varepsilon(b^2a_{ij}-b_ib_j), \ \ \ s_j=0,\label{45}
\end{eqnarray}
where $\varepsilon=\varepsilon(x)$ is a scalar function, and
$\phi=\phi(s)$ satisfies
\begin{eqnarray}
\Phi=-2(n+1)k\frac{\phi\Delta^2}{b^2-s^2},
\end{eqnarray}
where $k$ is a constant. In this case, $\textbf{S}=(n+1)cF$ with
$c=k\varepsilon$.\\
(b) $\beta$ satisfies
\begin{eqnarray}
r_{ij}=0,\hspace{.5cm}s_j=0.\label{44}
\end{eqnarray}
In this case, $\textbf{S}=0$.}
\end{lem}

%--------------------------------------------------------------------------------------------------------------------
\section{Proof of Theorem \ref{mainthm}}
%--------------------------------------------------------------------------------------------------------------------
\begin{lem}
Let $F=\alpha\phi(s)$, $s=\beta/\alpha$, be a non-Randers type $(\alpha,\beta)$-metric on a manifold $M$ of dimension $n\geq 3$. Suppose that $F$ has vanishing S-curvature. Then the following hold
\begin{eqnarray}
&&y_is^i_0=0,\label{L1}\\
&&y_is^i_{0|0}=0,\label{L2}\\
&&y_ib^js^i_{j|0}=\phi(\phi-s\phi')s^j_0s_{j0},\label{L3}
\end{eqnarray}
where $y_i:=g_{ij}y^j$.
\end{lem}
\begin{proof}
The following holds
\begin{eqnarray}
g_{ij}=\rho a_{ij}+\rho_0 b_ib_j+\rho_1(b_i\alpha_j+b_j\alpha_i)+\rho_2\alpha_i\alpha_j,
\end{eqnarray}
where $\alpha_i:=\alpha^{-1}a_{ij}y^j$ and
\begin{eqnarray}
&&\rho:=\phi(\phi-s\phi'),\\
&&\rho_0:=\phi\phi''+\phi'\phi',\\
&&\rho_1:=-\Big[s(\phi\phi''+\phi'\phi')-\phi\phi'\Big],\\
&&\rho_2:=s\Big[s(\phi\phi''+\phi'\phi')-\phi\phi'\Big].
\end{eqnarray}
Then
\begin{eqnarray}
y_i:=\rho\bar{y}_i+\rho_0 b_i\beta+\rho_1(b_i\alpha+s\bar{y_i})+\rho_2\bar{y}_i,\label{a1}
\end{eqnarray}
where $\bar{y}_i:=a_{ij}y^j$. Since $\bar{y}_is^i_0=0$ then by (\ref{44}) we get $b_is^i_0=0$. Thus (\ref{a1}) implies that
\begin{eqnarray}
y_is^i_0=0.\label{a2}
\end{eqnarray}
Since $y_{i|0}=0$, then by (\ref{a2}) it follows that
\begin{eqnarray}
y_is^i_{0|0}=0.
\end{eqnarray}
By $s_j=b^js^i_{j}=0$,  we have
\begin{eqnarray}
0=(b^js^i_{j})|_0=b^{j}_{|0}s^i_{j}+b^js^i_{j|0}=(r^j_0+s^j_0)s^i_j+b^js^i_{j|0}\label{0049}
\end{eqnarray}
or equivalently
\begin{eqnarray}
b^js^i_{j|0}=-s^j_0s^i_{j}.\label{0050}
\end{eqnarray}
By (\ref{a1}) and (\ref{0050}), we get
\begin{eqnarray}
y_ib^js^i_{j|0}=-(\rho+\rho_1s+\rho_2)s^j_0s^0_{j}=(\rho+\rho_1s+\rho_2)s^j_0s_{j0}.
\end{eqnarray}
Since $\rho_1s+\rho_2=0$, then
\begin{eqnarray}
y_ib^js^i_{j|0}=\rho s^j_0s_{j0}=\phi(\phi-s\phi')s^j_0s_{j0}.
\end{eqnarray}
This completes the proof.
\end{proof}

\bigskip

\begin{lem}
Let $F=\alpha\phi(s)$, $s=\beta/\alpha$, be a non-Randers type  $(\alpha,\beta)$-metric on a manifold $M$ of dimension $n\geq 3$. Suppose that $F$ has  vanishing S-curvature.   Then the following hold
 \begin{eqnarray}
 &&b^jb^kb^lL_{jkl}=0,\label{L4}\\
 &&b^iJ_i=0.\label{L5}
 \end{eqnarray}
\end{lem}
\begin{proof}
Since $F$ has vanishing S-curvature, then (\ref{G1}) reduces to following
\begin{eqnarray}
G^i=G^i_{\alpha}+\alpha Q s^i_0.\label{045}
\end{eqnarray}
Taking  third order vertical  derivations of  (\ref{045}) with respect to $y^j$, $y^l$ and $y^k$ yields
\begin{eqnarray}
\nonumber B^i_{jkl}=\!\!\!\!&&\!\!\!\! s^i_l\Big[Q\alpha_{jk}+Q_{k}\alpha_j+Q_{j}\alpha_k+\alpha Q_{jk}\Big]\\
\!\!\!\!&+&\!\!\!\! \nonumber s^i_j\Big[Q\alpha_{lk}+Q_{k}\alpha_l+Q_{l}\alpha_k+\alpha Q_{lk}\Big]\\
\!\!\!\!&+&\!\!\!\! \nonumber s^i_k\Big[Q\alpha_{jl}+Q_{j}\alpha_l+Q_{l}\alpha_j+\alpha Q_{jl}\Big]\\
\!\!\!\!&+&\!\!\!\! \nonumber s^i_0\Big[\alpha_{jkl}Q+\alpha_{jk}Q_l+\alpha_{lk}Q_j+\alpha_{lj}Q_k\\
\!\!\!\!&+&\!\!\!\! \alpha Q_{jkl}+\alpha_lQ_{jk}+\alpha_jQ_{lk}+\alpha_kQ_{jl}\Big]\label{046}
\end{eqnarray}
Multiplying (\ref{046}) with $y_i$ and using (\ref{L1}) implies that
\begin{eqnarray}
\nonumber -2L_{jkl}=\!\!\!\!&&\!\!\!\! y_is^i_{\ l}\Big[Q\alpha_{jk}+Q_{k}\alpha_j+Q_{j}\alpha_k+\alpha Q_{jk}\Big]\\
\!\!\!\!&+&\!\!\!\! \nonumber y_is^i_{\ j}\Big[Q\alpha_{lk}+Q_{k}\alpha_l+Q_{l}\alpha_k+\alpha Q_{lk}\ \Big]\\
\!\!\!\!&+&\!\!\!\! y_is^i_{\ k}\Big[Q\alpha_{jl}+Q_{j}\alpha_l+Q_{l}\alpha_j+\alpha Q_{jl}\ \Big].\label{t1}
\end{eqnarray}
By  (\ref{44}), we have $s_j=b^js_{ij}=0$. Then, multiplying (\ref{t1}) with $b^jb^kb^l$ yields (\ref{L4}).  By (\ref{05}) and (\ref{44}),  we get (\ref{L5}).
\end{proof}
\bigskip

\begin{lem}\label{lem1}
Let $(M,F)$ be a generalized P-reducible Finsler manifold. Then the Matsumoto torsion of $F$ satisfy in following
\begin{eqnarray}
M_{ijk|s}y^s=\lambda(x,y)M_{ijk}.\label{MM}
\end{eqnarray}
\end{lem}
\begin{proof}
Let $F$ be a generalized P-reducible metric
\begin{equation} \label{P2}
L_{ijk}=\lambda C_{ijk}+ a_ih_{jk}+a_jh_{ki}+a_kh_{ij}.
\end{equation}
Contracting (\ref{P2}) with $g^{ij}:=(g_{ij})^{-1}$  and using the relations $g^{ij}h_{ij}=n-1$ and $g^{ij}(a_ih_{jk})=g^{ij}(a_jh_{ik})=a_k$ implies that
\begin{equation} \label{P3}
J_k=\lambda I_k+ (n+1)a_k.
\end{equation}
Then
\begin{equation} \label{P4}
a_i=\frac{1}{n+1}J_i-\frac{\lambda}{n+1} I_i.
\end{equation}
Putting (\ref{P4}) in (\ref{P2}) yields
\begin{eqnarray}
\nonumber L_{ijk}= \lambda C_{ijk}\!\!\!\!&+&\!\!\!\!\ \frac{1}{n+1}\big\{J_ih_{jk}+J_jh_{ki}+J_kh_{ij}\big\}\\ \!\!\!\!&-&\!\!\!\!\ \frac{\lambda}{n+1}\big\{I_ih_{jk}+I_jh_{ki}+I_kh_{ij}\big\}.\label{P5}
\end{eqnarray}
By simplifying  (\ref{P5}), we get (\ref{MM}).
\end{proof}

\bigskip

\begin{lem}\label{lem4}
Let $F=\alpha\phi(s)$, $s=\beta/\alpha$, be a non-Randers type  $(\alpha,\beta)$-metric on a manifold $M$ of dimension $n\geq 3$. Suppose that $F$  is a generalized P-reducible metric with vanishing S-curvature. Then $F$ is a P-reducible metric.
\end{lem}
\begin{proof}
Let $F$ be a generalized P-reducible metric. By Lemma \ref{lem1}, we have
\begin{eqnarray}
L_{ijk}-\frac{1}{n+1}(J_ih_{jk}+J_jh_{ik}+J_kh_{ij})=\lambda\Big[C_{ijk}
-\frac{1}{n+1}(I_ih_{jk}+I_jh_{ik}+I_kh_{ij})\Big].\label{1}
\end{eqnarray}
Contracting (\ref{1}) with $b^ib^jb^k$  and using (\ref{L4}) and (\ref{L5}), implies that
\begin{eqnarray}
\lambda\Big[b^ib^jb^kC_{ijk}-\frac{3}{n+1}(b^iI_i)(b^jb^kh_{jk})\Big]=0.\label{4}
\end{eqnarray}
By (\ref{4}), we get two cases as follows:\\\\
\textbf{Case (1):}  $\lambda=0$. In this case, $F$ reduces to a P-reducible metric.\\\\
\textbf{Case (2):} $\lambda\neq 0$. In this case, by (\ref{4}) we get
\begin{eqnarray}
b^ib^jb^kC_{ijk}=\frac{3}{n+1}(b^iI_i)(b^jb^kh_{jk}).\label{5}
\end{eqnarray}
Multiplying (\ref{m1}) with $b^ib^jb^k$ gives
\begin{eqnarray}
b^ib^jb^k C_{ijk}=\frac{3p}{n+1}(b^iI_i)(b^j b^k h_{jk})+\frac{q}{\|\textbf{I}\|^2}(b^iI_i)^3.\label{6}
\end{eqnarray}
By (\ref{5}) and (\ref{6}), it follows that
\begin{eqnarray}
\frac{3q}{n+1}(b^iI_i)\Bigg[b^j b^k h_{jk}-\frac{(n+1)(b^mI_m)^2}{3\|\textbf{I}\|^2}\Bigg]=0.\label{6.5}
\end{eqnarray}
By (\ref{6.5}), we get three cases as follows:\\\\
\textbf{Case (2a):} Let $b^iI_i=0$. By a direct computation, we can obtain a formula for the mean Cartan torsion of $(\alpha,\beta)$-
metrics as follows
\begin{eqnarray}
I_i=-\frac{\Phi(\phi-s \phi')}{2\Delta\phi\alpha^2}(\alpha b_i-s y_i).\label{08}
\end{eqnarray}
If $b^iI_i=0$, then by contracting (\ref{08}) with $b^i$, we get
\begin{eqnarray}\label{09}
\frac{\Phi(\phi-s \phi')}{2\Delta\phi\alpha^3}( b^2\alpha^2-\beta^2)=0.
\end{eqnarray}
By (\ref{09}),  we have $\Phi=0$ or $\phi-s \phi'=0$ which implies that ${\bf I}=0$ and  then $F$ is a Riemannian metric. This is a  contradiction with our assumptions.  \\\\
\textbf{Case (2b):} Suppose that the following holds
\be
b^j b^k h_{jk}-\frac{n+1}{3\|\textbf{I}\|^2}(b^iI_i)^2=0.\label{al1}
\ee
Since $h_{jk}=g_{jk}-F^{-2}g_{jm}g_{kl}y^my^l$, then
\be
b^jb^k h_{jk}=b^jb^k g_{jk}-\frac{1}{F^2}\big(g_{jk}b^jb^k\big)^2\label{al2}
\ee
By (\ref{al1}) and (\ref{al2}), we obtain
\begin{eqnarray}
b^jb^k\Bigg[ g_{jk}-\frac{n+1}{3\|\textbf{I}\|^2}\ I_jI_k \Bigg]=\Bigg[\frac{1}{F}g_{jk}b^jb^k\Bigg]^2.\label{al3}
\end{eqnarray}
Since $y^iI_i=0$, then by (\ref{al3}), we get
\begin{eqnarray}
\Bigg[\Big(g_{ij}-\frac{(n+1)I_iI_j}{3\|\textbf{I}\|^2}\Big)b^i\frac{y^j}{F}\Bigg]^2=\Bigg[\Big(g_{ij}
-\frac{(n+1)I_iI_j}{3\|\textbf{I}\|^2}\Big)b^ib^j\Bigg].\label{40}
\end{eqnarray}
Put
\[
G_{ij}:=g_{ij}-\frac{n+1}{3\|\textbf{I}\|^2}I_iI_j.
\]
It follows from (\ref{40})  that
\begin{eqnarray}
\Big[G_{ij}b^i\frac{y^j}{F}\Big]^2=G_{ij}b^ib^j.\label{7}
\end{eqnarray}
Since $G_{ij}y^iy^j=F^2$, then  (\ref{7}) implies that
\begin{eqnarray}
\Big[G_{ij}b^i\frac{y^j}{F}\Big]^2=\Big[G_{ij}b^ib^j\Big]\Big[G_{ij}\frac{y^i}{F}\frac{y^j}{F}\Big].\label{8}
\end{eqnarray}
By Cauchy-Schwartz type inequality and (\ref{8}), we have
\be
b^i=k\frac{y^i}{F},\label{k}
\ee
where $k$ is a real constant. Multiplying (\ref{k}) with $b_i$ and $\bar{y_i}$, respectively,  implies that
\be
F=k\beta/b^2 \ \ \ \  \textrm{and} \ \ \  F=k\alpha^2/\beta.\label{k2}
\ee
By (\ref{k2}), it follows that $(b^2-s^2)\alpha^2=0$ which is a contradiction.\\\\
\textbf{Case (2c):} If $q=0$  then $p=1$ and from (\ref{m1}), it results that $F$ is C-reducible. In any cases, $F$ is a P-reducible Finsler metric.
\end{proof}

\bigskip

Now, we are going to consider P-reducible $(\alpha,\beta)$-metrics with vanishing S-curvature.

\begin{lem}\label{lem5}
Let $F=\alpha\phi(s)$, $s=\beta/\alpha$, be a non-Randers type  $(\alpha,\beta)$-metric on a manifold $M$ of dimension $n\geq 3$. Suppose that $F$  is a P-reducible metric with vanishing S-curvature. Then $F$ reduces to a Berwald metric or C-reducible metric.
\end{lem}
\begin{proof}
The Landsberg curvature of an $(\alpha, \beta)$-metric is given by
\begin{equation}
L_{ijk}=\frac{-\rho}{6\alpha^5}\Big\{h_ih_jC_k+h_jh_kC_i+h_ih_kC_j+3E_iT_{jk}+3E_jT_{ik}+3E_kT_{ij}\Big\},\label{Landsberg}
\end{equation}
where
\begin{eqnarray}
h_i\!\!\!\!&:=&\!\!\!\! \alpha  b_i-s\bar{y}_i,\label{hi} \\
T_{ij}\!\!\!\!&:=&\!\!\!\! \alpha^2a_{ij}-\bar{y}_i\bar{y}_j,\label{tttt}\\
\nonumber C_i\!\!\!\!&:=&\!\!\!\!  (X_4r_{00}+Y_4\alpha s_0)h_i +3\Lambda D_i, \\
\nonumber E_i\!\!\!\!&:=&\!\!\!\!  (X_6r_{00}+Y_6\alpha s_0)h_i +3\mu D_i,\\
\nonumber D_i\!\!\!\!&:=&\!\!\!\!  \alpha^2(s_{i0}+\Gamma r_{i0}+\Pi \alpha s_i)-(\Gamma r_{00}+\Pi \alpha s_0)\bar{y}_i
\\
\nonumber X_4\!\!\!\!&:=&\!\!\!\! \frac{1}{2\Delta^2}\Big\{ -2\Delta Q'''+3(Q-sQ')Q''+3(b^2-s^2)(Q'')^2\Big\}, \\
\nonumber X_6\!\!\!\!&:=&\!\!\!\! \frac{1}{2\Delta^2}\Big\{(Q-sQ')^2+\big[2(s+b^2Q)-(b^2-s^2)(Q-sQ')\big]Q''\Big\},\\
\nonumber Y_4\!\!\!\!&:=&\!\!\!\!  -2QX_4+\frac{3Q'Q''}{\Delta},\\
\nonumber Y_6\!\!\!\!&:=&\!\!\!\!  -2QX_6+\frac{(Q-sQ')Q'}{\Delta},\\
\nonumber \Lambda\!\!\!\!&:=&\!\!\!\!  -Q'',\ \ \  \mu:=-\frac{1}{3}(Q-sQ'),\\
\nonumber \Gamma\!\!\!\!&:=&\!\!\!\!\frac{1}{\Delta},\  \ \   \Pi:=\frac{-Q}{\Delta}.
\end{eqnarray}
For more details see  \cite{Shen}. Since $r_{ij}=0$ and $s_i=0$,  then (\ref{meanlanndsberg}) and (\ref{Landsberg})  reduce to following
\begin{eqnarray}
J_i\!\!\!\!&=&\!\!\!\! - \frac{\Phi}{2\alpha\Delta} s_{i0},\label{meanlanndsberg2}\\
L_{ijk}\!\!\!\!&=&\!\!\!\! V_{ij}s_{k0}+V_{jk}s_{i0}+V_{ki}s_{j0},\label{Landsberg1}
\end{eqnarray}
where
\[
V_{ij}:=\frac{\rho}{2\alpha^3}\Big[Q''h_ih_j+(Q-sQ')T_{ij}\Big].
\]
We shall divide the problem into two cases: (a) $s_{i0}= 0$ and (b) $s_{i0}\neq 0$.\\\\
{\bf Case (a):} Let $s_{i0}= 0$. In this case, by (\ref{meanlanndsberg2}) and  (\ref{Landsberg1}), $F$ reduces to a Landsberg metric. By Shen' Theorem in \cite{Shen}, $F$ reduces to a Berwald metric.\\\\
{\bf Case (b):} Let $s_{i0}\neq 0$. Then by  (\ref{meanlanndsberg2}) and (\ref{Landsberg1}), we have
\begin{eqnarray}
L_{ijk}= Z_{ij}J_k+Z_{jk}J_i+Z_{ki}J_j,\label{yy}
\end{eqnarray}
where $Z_{ij}:= -\frac{2\alpha\Delta}{\Phi}V_{ij}$. Thus the Landsberg curvature of an  $(\alpha, \beta)$-metric with vanishing S-curvature satisfies (\ref{yy}). Put
\[
A:=-\frac{\Delta\rho(Q-sQ')}{\Phi}, \ \ \ \ B:=-\frac{\Delta\rho Q''}{\Phi}.
\]
Then by putting (\ref{hi}) and (\ref{tttt}) in the formula of $Z_{ij}$ it follows that
\begin{eqnarray}
Z_{ij}=A a_{ij}+B b_ib_j-sB(b_i\alpha_j+b_j\alpha_i)-(A-s^2B)\alpha_i\alpha_j.\label{Z}
\end{eqnarray}
By assumption, $F$ is P-reducible
\begin{eqnarray}
L_{ijk}=\frac{1}{n+1}(J_ih_{jk}+J_jh_{ik}+J_kh_{ij}),\label{q1}
\end{eqnarray}
where the angular metric $h_{ij}:= g_{ij} -F_{y^i}F_{y^j}$ is given by following
\[
h_{ij} =\phi [  \phi -s \phi']a_{ij}+\phi \phi''\ b_ib_j  - s \phi\phi'' [b_i\alpha_j + b_j \alpha_i ] -\big[  \phi (  \phi - s \phi')  -s^2 \phi\phi''\big]\ \alpha_i \alpha_j.\label{angular}
\]
By (\ref{yy}) and (\ref{q1}), we obtain
\begin{equation}
(Z_{ij}-\frac{1}{n+1}h_{ij})J_k+(Z_{jk}-\frac{1}{n+1}h_{jk})J_i+(Z_{ik}-\frac{1}{n+1}h_{ik})J_j=0.\label{q2}
\end{equation}
Since $\alpha_i s^i_0=0$ and $b_is^i_0=0$, then we have
\begin{eqnarray*}
&&s^i_0s^j_0Z_{ij}=-\frac{\Delta\rho}{\Phi}(Q-sQ')s^m_0s_{m0},\\
&&s^i_0s^j_0h_{ij}=\phi [  \phi -s \phi']s^m_0s_{m0},\\
&&s^i_0J_i= - \frac{\Phi}{2\alpha\Delta}s^m_0s_{m0}.
\end{eqnarray*}
Therefore, contracting (\ref{q2}) with $s^i_0s^j_0s^k_0$ implies that
\begin{equation}
\frac{1}{n+1}\phi [  \phi -s \phi']=A.\label{q3}
\end{equation}
By (\ref{q3}), it follows that
\begin{eqnarray}
Z_{ij}-\frac{1}{n+1}h_{ij}=\chi \Big[b_ib_j- s(b_i\alpha_j + b_j \alpha_i )+ s^2\alpha_i \alpha_j\Big],\label{q4}
\end{eqnarray}
where
\[
\chi:=B-\frac{1}{n+1}\phi\phi''.
\]
Since $J_i\neq0$ and $b^mJ_m=0$, then by multiplying (\ref{q2}) with $b^ib^j$, we get
\begin{equation}
b^ib^j(Z_{ij}-\frac{1}{n+1}h_{ij})=0.\label{A00}
\end{equation}
By contracting (\ref{q4}) with $b^ib^j$ and considering (\ref{A00}),  it follows that
\begin{eqnarray}
\chi =0.\label{A000}
\end{eqnarray}
(\ref{q3}) and (\ref{A000}) imply that
\begin{eqnarray}
\frac{1}{n+1}\phi [  \phi -s \phi']\!\!\!\!&=&\!\!\!\! -\frac{\Delta\rho}{\Phi}(Q-sQ'),\label{e1} \\
\frac{1}{n+1}\phi \phi''\!\!\!\!&=&\!\!\!\! -\frac{\Delta\rho}{\Phi} Q''.\label{e2}
\end{eqnarray}
By (\ref{e1}) and (\ref{e2}), we obtain
\be
 \phi -s \phi'=c(Q-sQ'),\label{e3}
\ee
where $c$ is a non-zero real constant. Solving (\ref{e3}) implies that
\be
Q=c_1\phi+c_2s,\label{e4a}
\ee
where $c_1\neq 0$ and $c_2$ are real constants. By  (\ref{e4a}), it follows that
\be
c_2s^2+2c_1s\phi+1=d\phi^2,\label{e4}
\ee
where $d$ is a real constant. We shall divide the problem into two cases: (b1) $d\neq 0$ and (b2) $d = 0$.\\\\
{\bf Subcase (b1):} If $d\neq0$, then by (\ref{e4})  we have
\be
\phi=\frac{c_1}{d}s+\sqrt{\big[(\frac{c_1}{d})^2+\frac{c_2}{d}\big]s^2+1}    \label{e6}
\ee
which is a Randers-type metric. This is  a contradiction.\\\\
{\bf Subcase (b2):} If $d=0$, then  (\ref{e4}) yields
\be
\phi=-\frac{1}{2c_1 s}+\frac{c_2}{2c_1}s\label{e5}
\ee
which is a Randers change of a Kropina metric. It is known that,  Kropina metrics are C-reducible. On the other hand, every Randers change of C-reducible metric is C-reducible \cite{M2}. Then the Finsler metric defined by (\ref{e5}) is C-reducible.
\end{proof}

\bigskip

\noindent
{\bf Proof of Theorem \ref{mainthm}:} Every two-dimensional Finsler surface is C-reducible. For  Finsler manifolds of dimension $n\geq 3$, by Lemmas \ref{lem4} and \ref{lem5}, we get the proof.
\qed

%-----------------------------------------------------------------------------------

\noindent
Akbar Tayebi and Hassan Sadeghi\\
Department of Mathematics, Faculty  of Science\\
University of Qom \\
Qom. Iran\\
Email:\ akbar.tayebi@gmail.com\\
Email:\ sadeghihassan64@gmail.com

\begin{thebibliography}{MaHo}
%------------------------------------------------------------------------------------------------------------------
\bibitem{Ber} L. Berwald, {\it On Cartan and Finsler geometries III, Two dimensional Finsler spaces with rectliniear extremals}, Ann of Math. {\bf 42}(1941), 84-112.
%----------------------------------------------------------------------------------------------------
\bibitem{Cheng} X. Cheng, {\it On $(\alpha,\beta)$-metrics of scalar flag curvature with constant S-curvature}, Acta. Math. Sinica, English Series, {\bf 26}(9) (2010), 1701-1708.
%----------------------------------------------------------------------------------------------------
\bibitem{CS} X. Cheng and Z. Shen, {\it Finsler Geometry, An Approach via Randers Spaces}, Springer-Verlag, 2012.
%----------------------------------------------------------------------------------
\bibitem{ChSh3} X. Cheng and Z. Shen, {\it A class of Finsler metrics with isotropic S-curvature}, Israel J. Math. {\bf 169}(2009), 317-340.
%------------------------------------------------------------------------------------------------------------------
\bibitem{M2} M. Matsumoto,  {\it Projective Randers change of P-reducible Finsler spaces}, Tensor N. S. {\bf 59}(1998), 6-11.
%----------------------------------------------------------------------------------
\bibitem{M3} M. Matsumoto,   {\it On Finsler spaces with Randers metric and special forms of important tensors},  J. Math. Kyoto Univ.  {\bf 14}(1974), 477-498.
%----------------------------------------------------------------------------------
 \bibitem{M} M. Matsumoto,  {\it Finsler spaces with the hv-curvature tensor $P_{hijk}$ of a special form},  Rep. Math. Phys.   {\bf 14}(1978), 1-13.
%----------------------------------------------------------------------------------
 \bibitem{Mat} M. Matsumoto,  {\it Theory of Finsler spaces with $(\alpha, \beta)$-metric}, Rep. Math. Phys. {\bf 31}(1992), 43-84.
%------------------------------------------------------------------------------------------------------------------
\bibitem{MaHo} M. Matsumoto and S. H\={o}j\={o}, {\it A conclusive theorem for C-reducible Finsler spaces}, Tensor. N. S. {\bf 32}(1978), 225-230.
%----------------------------------------------------------------------------------
 \bibitem{MShi} M. Matsumoto and H. Shimada,  {\it On Finsler spaces with the curvature tensors $P_{hijk}$ and $S_{hijk}$ satisfying special conditions},  Rep. Math. Phys.   {\bf 12}(1977), 77-87.
%---------------------------------------------------------------------------------------------------------------------
\bibitem{Numata} S. Numata, {\it On the torsion tensors $R_{jhk}$ and $P_{hjk}$ of Finsler spaces with a metric
$ds = (g_{ij}(dx)dx^idx^j)^2 + b_i(x)dx^i$"},  Tensor (N.S.), {\bf 32}(1978), 27-32.
%-----------------------------------------------------------------------------------
\bibitem{Shvol} Z. Shen, {\it  Volume comparison and its applications in Riemann-Finsler geometry}, Advances. Math. {\bf 128}(1997), 306-328.
%------------------------------------------------------------------------------------------------------------------
\bibitem{Shen} Z. Shen, {\it On a class of Landsberg metrics in Finsler geometry}, Canad. J. Math. {\bf 61}(2009), 1357-1374.
%-----------------------------------------------------------------------------------
\bibitem{Tak} Y. Takano, {\it On the theory of fields  in Finsler spaces},  Intern.  Symp.  Relativity  and  Unified  Field  Theory,   Calcutta, 1975.
%-----------------------------------------------------------------------------------
\bibitem{TN} A. Tayebi and B. Najafi, {\it  On isotropic Berwald metrics}, Ann. Polon. Math. {\bf 103}(2012), 109-121.
%-----------------------------------------------------------------------------------
\bibitem{TPN} A. Tayebi, E. Peyghan and B. Najafi, {\it   On Semi-C-reducibility of  $(\alpha, \beta)$-metrics}, Int. J. Geom. Meth. Modern. Phys, {\bf 9}(4) (2012), 1250038.
%-----------------------------------------------------------------------------------
\bibitem{TS} A. Tayebi and H. Sadeghi, {\it On Cartan torsion of Finsler metrics}, Publ. Math. Debrecen, {\bf 82}(2) (2013), 461-471.
\end{thebibliography}
\end{document}